\documentclass[12pt,reqno]{amsart}
\usepackage{amssymb}
\usepackage{mathrsfs}\overfullrule=5pt
\usepackage{amsmath,amssymb,graphicx}
\usepackage[colorlinks=true,urlcolor=red, citecolor=blue]{hyperref}
\usepackage{xcolor}
\makeatletter

\newcommand{\Rmnum}[1]{\expandafter\@slowromancap\romannumeral #1@}
\makeatother

\newtheorem{thm}{Theorem}[section]
\newtheorem{lem}[thm]{Lemma}

\newtheorem{prop}[thm]{Proposition}

\theoremstyle{definition}

\theoremstyle{remark}
\newtheorem{rem}[thm]{Remark}

\DeclareMathOperator{\supp}{supp}
\numberwithin{equation}{section}

\newcommand{\Z}{\mathbb{Z}}
\newcommand{\N}{\mathbb{N}}
\newcommand{\R}{\mathbb{R}}
\newcommand{\p}{\mathbb{P}}
\begin{document}
\title[Mean Li-Yorke chaos along polynomials and prime numbes]{Mean Li-Yorke chaos along polynomials of several variables and prime numbers}
\author[K. Liu]{Kairan Liu}
\address{Kairan Liu:
Department of Mathematics, University of Science and Technology of China, Hefei, Anhui 230026, China}
\email{lkr111@mail.ustc.edu.cn}


\date{\today}

\begin{abstract}
In this paper, for any given polynomial, by analyzing the limiting behavior of ergodic averages along polynomials of several variables and prime numbers, we prove that for a topology dynamical system, positive entropy implies mean Li-Yoke chaos along non-constant polynomials of several variables and prime numbers.
\end{abstract}

\maketitle

\section{Introduction}
As an important concept representing complexity of a dynamical system, the notion of chaos was first introduced by Li and Yorke in 1975 (\cite{Li-Yorke}), and has attracted a lot of attention. Other different versions of chaos, such that Devaney chaos, positive entropy and weakly mixing (see for example \cite{Devaney,Adler-Konheim-McAndrew,Furstenberg} for details) were proposed over the past few decades. The implication among those chaos became a central topic as well.

In 1991, Iwanik showed that weak mixing implies Li-Yorke chaos (\cite{Iwanik}). Latter, Huang and Ye (\cite{Huang-Ye}) proved that Devaney chaos implies Li-Yorke one. In the same year, it was showed by Blanchard et al. that positive entropy implies Li-Yorke chaos (see \cite{Blanchard-Glasner-Kolyada}) and for amenable group case see \cite{Huang-Xu-Yi,Kerr-Li,Wang-Zhang}. Moreover, the result also holds for sofic group actions by Kerr and Li (\cite{Kerr-Li2}). In \cite{Downarowicz}, Downarowicz proved that positive topological entropy implies mean Li-Yorke chaos, see \cite{Huang-Li-Ye-Zhou} for another approach. See a recent survey \cite{Li-Ye} and references therein for more results and details.

Ergodic theory has a long history of interaction with other mathematical fields and in particular with combinatorics and number theory. The seminal work of H. Furstenberg \cite{Furstenberg2}, where an ergodic proof of the theorem of Szemer\'edi \cite{Szemeredi} on arithmetic progressions was given, linked problems in ergodic theory, combinatorics, and number theory, and provided an ideal ground for cross-fertilization. Many
researchers were interested in multiple ergodic averages. To get manageable problems, one typically restricts the class of eligible sequences and usually assumes that they are polynomial sequences, sequences arising from smooth functions, sequences related to the
prime numbers, or random sequences of integers. One can see \cite{Rosenblatt-Wierdl} and references therein for more results. In \cite{Leibman}, the author obtained the pointwise ergodic theorem for polynomial actions of $Z^d$ by translations on a nilmanifold.

Inspired by the previous works, our aim in this paper is to investigate relationships between positive entropy and mean Li-Yorke chaos along polynomials of several variables and prime numbers. Precisely, throughout this paper, by a \emph{topological dynamical system} (TDS for short), we mean a pair $(X,T)$, where $X$ is a compact metric space and $T:X\to X$ is a homeomorphism, and by a \emph{probability space}, we mean a triple $(X,\mathscr{B}_X,\mu)$, where $(X,\mathscr{B}_X)$ is a standard Borel space and $\mu$ is a probability measure on $(X,\mathscr{B}_X)$. For a probability space $(X,\mathscr{B}_X,\mu)$, if $T:X\to X$ is an invertible measure preserving transformation, by a \emph{measure preserving system}, we mean the quadruple $(X,\mathscr{B}_X,\mu,T)$.

\begin{thm}\label{mainthem1}Let $s\in\N$, $P:\Z^s\to\Z$ be a non-constant polynomial, $\{\Phi_n\}_{n=1}^{\infty}$ a F$\phi$lner sequence of $\Z^s$ and $(X,T)$ be a TDS with $h_{top}(X,T)>0$. Then there exists a Cantor subset $C$ of $X$, such that for every distinct $x,y\in C$ one has
$$\left\{
\begin{aligned}
\limsup\limits_{N\to\infty}\frac{1}{|\Phi_N|}\sum\limits_{u\in\Phi_N}\rho(T^{P(u)}x,T^{P(u)}y)>0, \\
\liminf\limits_{N\to\infty}\frac{1}{|\Phi_N|}\sum\limits_{u\in\Phi_N}\rho(T^{P(u)}x,T^{P(u)}y)=0.
\end{aligned}
\right.
$$
\end{thm}

For any $N\in \N$, let $\pi(N)$ denote the number of primes less than or equal to $N$. Authors of \cite{Frantzikinakis-Host-Kra}, studied multiple recurrence and convergence for sequences related to the prime numbers. Based on those results, we have the following result.
\begin{thm}\label{mainthm2}Let $s\in\N$, $P:\Z^s\to\Z$ be a non-constant polynomial, and $(X,T)$ be a TDS with $h_{top}(X,T)>0$. Then there exists a Cantor subset $C$ of $X$, such that for every distinct $x,y\in C$ we have
$$\left\{
		\begin{aligned}
		\limsup\limits_{N\to\infty}\frac{1}{\pi(N)^s}\sum\limits_{\mbox{\tiny$\begin{array}{c}
1\leq p_1,\dotsc, p_s\leq N,\\
p_1,\dotsc,p_s\in\p\end{array}$}}\rho(T^{P(p_1,\dotsc,p_s)}x,T^{P(p_1,\dotsc,p_s)}y)>0, \\
        \liminf\limits_{N\to\infty}\frac{1}{\pi(N)^s}\sum\limits_{\mbox{\tiny$\begin{array}{c}
1\leq p_1,\dotsc, p_s\leq N,\\
p_1,\dotsc,p_s\in\p\end{array}$}}\rho(T^{P(p_1,\dotsc,p_s)}x,T^{P(p_1,\dotsc,p_s)}y)=0.
		\end{aligned}
		\right.
$$
\end{thm}

This paper is organized as follows. In Section 2, we list basic notions and results needed in our argument. In Section 3, we study some propositions of ergodic average along a no-constant polynomial. In Section 4, we prove Theorem \ref{mainthem1}. Finally, we consider mean Li-Yorke chaos along polynomials of prime numbers and prove Theorem \ref{mainthm2} in Section 5.
\section{Preliminaries}
In this section, we will review some basic notions and fundamental properties that will be used latter.
\subsection{Amenable groups and Banach density in an amenable group}
Recall that a countable discrete group $G$ is called \emph{amenable} if there exists a sequence of finite subsets $\Phi_n\subset G$ such that $\lim\limits_{n\to+\infty}\frac{|g\Phi_n\triangle \Phi_n|}{|\Phi_n|}=0$ holds for every $g\in G$, and we say that such $\{\Phi_n\}_{n=1}^{+\infty}$ is a \emph{F$\phi$lner sequence of $G$}. It is clear that $\Z^s$ is an amenable group for any $s\in N$.

Let $G$ be an amenable group, and $\{\Phi_n\}_{n=1}^{+\infty}$ be a F$\phi$lner sequence of $G$. For any given subset $F$ of $G$, we denote the \emph{upper density of $F$ with respect to $\{\Phi_n\}_{n=1}^{+\infty}$} by
$$\overline{d}_{\{\Phi_n\}}(F):=\limsup\limits_{n\to+\infty}\frac{\big\vert F\cap\Phi_n\big\vert}{\vert\Phi_n\vert}.$$
The \emph{upper Banach density of $F$} is defined by
$$d^*(F):=\sup\big\{\overline{d}_{\{\Phi_n\}}(F):\ \{\Phi_n\}_{n=1}^{+\infty}\ is\ a\ F\phi lner\ sequence \big\}.$$

For $G=\Z^s$, $s\in\N$, the above definition differs from original definition of upper Banach density where the supremum was taken only over intervals instead of arbitrary F$\phi$lner sequence. However the authors of \cite{Beiglock-Bergelson-Fish} proved that the two notions are equivalent(see \cite[Lemma 3.3]{Beiglock-Bergelson-Fish}).
\begin{lem}\label{density} Let $G$ be an amenable group and $\{\Phi_n\}_{n=1}^{+\infty}$ be a F$\phi$lner sequence. Then there exists a sequence $\{t_{n}\}_{n=1}^{+\infty}$ such that $d^*(F)=\overline{d}_{\{\Phi_nt_n\}}(F)$.
\end{lem}

\subsection{disintegration of measure}For a TDS $(X,T)$, we denote the collection of all Borel probability measures of $X$ by $\mathcal{M}(X)$, the collection of all $T$-invariant Borel probability measures of $X$ by $\mathcal{M}(X,T)$, and the collection of all ergodic measures of $(X,T)$ by $\mathcal{M}^e(X,T)$. We now recall the main results and properties of condition expectation and disintegration of
measures. We refer to \cite[Chapter 5]{Einsiedler-Ward} for more details.

Let $(X,\mathscr{B}_X,\mu)$ be a probability space, and $\mathscr{A}\subseteq\mathscr{B}_X$ a sub-$\sigma$-algebra. Then there is a map
$$E(\cdot|\mathscr{A}):L^1(X,\mathscr{B}_X,\mu)\to L^1(X,\mathscr{A},\mu)$$
called the \emph{conditional expectation}, that satisfies the following properties.
\begin{enumerate}
\item For $f\in L^1(X,\mathscr{B}_X,\mu)$, the image function $E(f|\mathscr{A})$ is characterized almost everywhere by the two properties:
\begin{itemize}
	\item $E(f|\mathscr{A})$ is $\mathscr{A}$-measurable;
	\item for any $A\in\mathscr{A}$, $\int_A E(f|\mathscr{A})d\mu=\int_A fd\mu$.
\end{itemize}
\item $E(\cdot|\mathscr{A})$ is a linear operator of norm 1. Moreover, $E(\cdot|\mathscr{A})$ is positive.
\item For $f\in L^1(X,\mathscr{B}_X,\mu)$ and $g\in L^{\infty}(X,\mathscr{A},\mu)$,
    $$E(g\cdot f|\mathscr{A})=g\cdot E(f|\mathscr{A})$$
    almost everywhere.
\item $\mathscr{A}'\subseteq\mathscr{A}$ is a sub-$\sigma$-algebra, then
    $$E\big(E(f|\mathscr{A})\big|\mathscr{A}'\big)=E(f|\mathscr{A}')$$
    almost everywhere.
\end{enumerate}
The following result is well-known (see e.g. \cite[ Theorem 14.26]{Glasner}, \cite[Section 5.2]{Einsiedler-Ward}).
\begin{thm}\label{function limits}Let $(X,\mathscr{B}_X,\mu)$ be a probability space. Suppose that $\{\mathscr{A}_n\}_{n=1}^{\infty}$ is a decreasing sequence (resp. an increasing sequence) of sub-$\sigma$-algebras of $\mathscr{B}_X$ and $\mathscr{A}=\bigcap\limits_{n\geq}\mathscr{A}_n$ (resp. $\mathscr{A}=\bigvee\limits_{n\geq1}\mathscr{A}_n$). Then for any $f\in L^1(\mu)$,
$$E(f|\mathscr{A}_n)\to E(f|\mathscr{A})$$
as $n\to\infty$ in $L^1(\mu)$ and $\mu$-almost everywhere.
\end{thm}

Let $(X,\mathscr{B}_X,\mu)$ be a Borel probability space, and $\mathscr{A}\subseteq\mathscr{B}_X$ a $\sigma$-algebra. Then $\mu$ can \emph{be disintegrated over $\mathscr{A}$} as
\[
\mu=\int_X \mu_x^{\mathscr{A}} d\mu(x)
\]
in the sense that for any $f\in L^1(X,\mathscr{B}_X,\mu)$, one has
\begin{equation}
E(f|\mathscr{A})(x)=\int f(y)d\mu_x^{\mathscr{A}}(y)\quad
\text{for }\mu\text{-a.e.\ }x\in X,\label{eq:Ef-mux}
\end{equation}
where $\mu_x^{\mathscr{A}}\in \mathcal{M}(X)$.

\subsection{Entropy}
Let $(X,\mathscr{B}_X,\mu,T)$ be a measure preserving system. A partition of $X$ is a cover of $X$, whose elements are pairwise disjoint. For a finite measurable partition $\alpha$, the \emph{measure-theoretic entropy of $\mu$ relative to $\alpha$}, denoted by $h_{\mu}(T,\alpha)$, is defined as
$$h_{\mu}(T,\alpha)=\lim_{n\to\infty}\frac{1}{n}H_{\mu}
\biggl(\bigvee_{i=0}^{n-1}T^{-i}\alpha\biggr),$$
where $H_{\mu}(\alpha)=-\sum\limits_{A\in\alpha}\mu(A)\log \mu(A)$. The \emph{measure-theoretic entropy of $\mu$} is defined as
$$h_{\mu}(X,T)=\sup\limits_{\alpha}h_{\mu}(T,\alpha),$$
where the supremum ranges over all finite measurable partitions of $X$.

The Pinsker $\sigma$-algebra of a system $(X,\mathscr{B}_X,\mu,T)$ is defined as
$$P_{\mu}(T)=\{A\in\mathscr{B}_X: h_{\mu}(T,\{A,X\backslash A\})=0.$$
The following Rokhlin-Sinai theorem identifies the Pinsker $\sigma$-algebra as the ``remote past'' of a generating partition (see \cite{Rokhlin-Sinai}).
\begin{thm}\label{Rokhlin-Sinai}Let $(X,\mathscr{B}_X,\mu,T)$ be a measure preserving system, then there exists a sub-$\sigma$-algebra $\mathscr{P}$ of $\mathscr{B}_X$ such that $T^{-1}\mathscr{P}\subset\mathscr{P}$, $\bigvee\limits_{k=0}^{\infty}T^k\mathscr{P}=\mathscr{B}_X$ and $\bigcap\limits_{n=0}^{\infty}T^{-k}\mathscr{P}=P_{\mu}(T).$
\end{thm}\section{Ergodic average along a no-constant polynomial}
In this section, by partially following from the arguments in \cite{Li-Qiao}, we study ergodic average along a non-constant polynomial. For a measure preserving system $(X,\mathcal{B}_X,\mu,T)$ and a measurable function $f$, we write $Tf$ for the function defined by $(Tf)(x)=f(Tx)$. In \cite{Leibman}, the author proved the following result.

\begin{thm}\label{Leibman}Let $(X,\mathscr{B}_X,\mu,T)$ be a measure preserving system, $s\in\N$ and $P:\Z^s\to\Z$ be a polynomial, then for any $f\in L^{\infty}(\mu)$ and any F$\phi$lner sequence $\{\Phi_N\}_{N=1}^{\infty}$ of $\Z^s$, the averages
$$\frac{1}{|\Phi_N|}\sum\limits_{u\in\Phi_{N}}T^{P(u)}f$$
converges in $L^2(\mu)$ as $N\to\infty$.
\end{thm}

\begin{prop}\label{disintegration} Let $s\in\N$, $P:\Z^s\to\Z$ be a non-constant polynomial and $\{\Phi_N\}_{N=1}^{\infty}$ be a F$\phi$lner sequence of $\Z^s$. Then for any measure preserving system $(X,\mathscr{B}_X,\mu,T)$, there exists a disintegration of $\mu$
$$\mu=\int \tau_xd\mu(x),$$
in the sense that, there exists a subset $X_0\in\mathscr{B}_X$ of $X$ with $\mu(X_0)=1$, and $\{N_i\}_{i=1}^{\infty}\subset\N$ such that for any $f\in C(X)$ and $x\in X_0$ one has
$$\lim\limits_{i\to\infty}\frac{1}{|\Phi_{N_i}|}\sum_{u\in\Phi_{N_i}}f(T^{P(u)}x)=\int fd\tau_x,$$
and
$$\int\int fd\tau_xd\mu(x)=\int fd\mu.$$
\end{prop}
\begin{proof}Let $\{g_n\}_{n=1}^{\infty}$ be a dense subset of $C(X)$. For $g_1$, by Theorem \ref{Leibman}, there exists an increasing sequence  $\{N_i^1\}_{i=1}^{\infty}\subset\N$ and $X_1\in\mathscr{B}_X$ with $\mu(X_1)=1$ such that for every $x\in X_1$ the averages
$$\frac{1}{|\Phi_{N_i^1}|}\sum\limits_{u\in\Phi_{N_i^1}}T^{P(u)}g_1(x)$$
converges as $i\to\infty$. Continuing this process, we can get $\{N_i^{k+1}\}_{i=1}^{\infty}\subset\{N_i^k\}_{i=1}^{\infty}$ and $X_{k+1}\subset X_k$ with $\mu(X_{k+1})=1$ such that for every $x\in X_{k+1}$ the averages
$$\frac{1}{|\Phi_{N_i^{k+1}}|}\sum\limits_{u\in\Phi_{N_i^{k+1}}}T^{P(u)}g_{k+1}(x)$$
converges as $i\to\infty$, for $k=1,2,\dotsc$. Let $N_i=N_i^i$ for every $i\in\N$ and $X_{\infty}=\bigcap_{i=1}^{\infty}X_i$, then $\mu(X_{\infty})=1$ and for every $x\in X_{\infty}$, $k\in\N$ the averages
$$\frac{1}{|\Phi_{N_i}|}\sum\limits_{u\in\Phi_{N_i}}T^{P(u)}g_k(x)$$
converges as $i\to\infty$. Let $X_0=X_{\infty}\cap \supp(\mu)$, then $\mu(X_0)=1$.

For every $f\in C(X)$, there exists $\{f_i\}_{i=1}^{\infty}\subset\{g_i\}_{i=1}^{\infty}$ such that $||f_j-f||_{L^{\infty}}\to 0$ as $j\to\infty$. For any $x_0\in X_0$ and $\varepsilon>0$ there exists $j_0$ such that $\Vert f_{j_0}-f\Vert_{L^{\infty}}\leq\varepsilon$. Since $f_j,f$ are continuous and $X_0\subseteq \supp(\mu)$, one has $T^nx_0\notin\{x\in X: \vert f_j(x)-f(x)\vert>\varepsilon\}$ for every $n\in\Z$. For $j_0$ there exists $i_0\in\N$ such that $\big\vert\frac{1}{|\Phi_{N_i}|}\sum\limits_{u\in\Phi_{N_i}}T^{P(u)}
f_{j_0}(x_0)-\frac{1}{|\Phi_{N_m}|}\sum\limits_{u\in\Phi_{N_m}}T^{P(u)}f_{j_0}(x_0)\big\vert<\varepsilon$ for every $i,m>i_0$. Then by triangle inequality one has
$$\bigg\vert\frac{1}{|\Phi_{N_i}|}\sum\limits_{u\in\Phi_{N_i}}T^{P(u)}
f(x_0)-\frac{1}{|\Phi_{N_m}|}\sum\limits_{u\in\Phi_{N_m}}T^{P(u)}f(x_0)\bigg\vert<3\varepsilon$$
for every $i,m>i_0$. Thus the average
$$\frac{1}{|\Phi_{N_i}|}\sum\limits_{u\in\Phi_{N_i}}T^{P(u)}f(x)$$
converges as $i\to\infty$ for every $f\in C(X)$ and $x\in X_0$.

For every $x\in X_0$, we set $L_x:C(X)\to\R$, $f\to\lim\limits_{i\to\infty}\frac{1}{|\Phi_{N_i}|}\sum\limits_{u\in\Phi_{N_i}}T^{P(u)}f(x)$. Then $L_x$ is a positive linear function with $L_x(1)=1$. By Riesz Representation Theorem, there exists $\tau_z\in \mathcal{M}(X)$ such that for any $f\in C(X)$ one has $L_x(f)=\int fd\tau_x$. And
\begin{align*}
\int\int fd\tau_xd\mu(x)&=\int\lim\limits_{i\to\infty}\frac{1}{|\Phi_{N_i}|}\sum\limits_{u\in\Phi_{N_i}}T^{P(u)}f_(x)d\mu(x)\\
&=\lim\limits_{i\to\infty}\frac{1}{|\Phi_{N_i}|}\sum\limits_{u\in\Phi_{N_i}}\int T^{P(u)}f_(x)d\mu(x)\\
&=\int fd\mu(x),
\end{align*}
the second equation holds because $f$ is bounded. This ends the proof.
\end{proof}

\begin{thm}\label{characteristic}Let $s\in\N$, $P:\Z^s\to\Z$ be a non-constant polynomial, and $(X,\mathscr{B}_X,$ $\mu,T)$ be a measure preserving system with the Pinsker $\sigma$-algebra $P_{\mu}(T)$. Then for any $f\in L^{\infty}(\mu)$, one has
$$\frac{1}{|\Phi_N|}\sum\limits_{u\in\Phi_N}T^{P(u)}f-\frac{1}{|\Phi_{N}|}\sum\limits_{u\in\Phi_N}T^{P(u)}E(f|P_{\mu}(T))\to 0,$$
as $N\to\infty$ in $L^2(\mu)$.
\end{thm}

\begin{proof}By Theorem \ref{Rokhlin-Sinai}, there exists a sub-$\sigma$-algebra $\mathscr{P}$ of $\mathscr{B}_X$ such that $T^{-1}\mathscr{P}\subset\mathscr{P}$, $\bigvee\limits_{k=0}^{\infty}T^k\mathscr{P}=\mathscr{B}_X$ and $\bigcap\limits_{n=0}^{\infty}T^{-k}\mathscr{P}=P_{\mu}(T)$. For any given $f\in L^{\infty}(\mu)$, firstly, we assume $f$ is $\mathscr{P}$-measurable. Since $f\in L^{\infty}(\mu)$, without lost of generality, we can assume $ \bigl\Vert f\bigr\Vert\leq 1$. Let $f^{\infty}$ denote $E(f|P_{\mu}(T))$ and $f^n$ denote $E(f|T^{-n}\mathscr{P})$ for $n\in\N$. For any $0<\varepsilon<\frac{1}{2}$, by Theorem \ref{function limits}, there exists $m\in\N$ such that $||f^m-f^{\infty}||_{L^2(\mu)}<\varepsilon$. Then for every $N\in\N$
\begin{footnotesize}
\begin{align}\label{e1}
&\bigg\Vert\frac{1}{|\Phi_N|}\sum\limits_{u\in\Phi_N}T^{P(u)}f-\frac{1}{|\Phi_N|}\sum\limits_{u\in\Phi_{N}}T^{P(u)}f^{\infty} \bigg\Vert_{L^2(\mu)}\nonumber\\
&\leq\bigg\Vert\frac{1}{|\Phi_N|}\sum\limits_{u\in\Phi_N}T^{P(u)}f-\frac{1}{|\Phi_{N}|}\sum\limits_{u\in\Phi_{N}}T^{P(u)}f^{m} \bigg\Vert_{L^2(\mu)}+
\bigg\Vert\frac{1}{|\Phi_N|}\sum\limits_{u\in\Phi_N}T^{P(u)}(f^m-f^{\infty}) \bigg\Vert_{L^2(\mu)}\nonumber\\
&\leq\bigg\Vert\frac{1}{|\Phi_N|}\sum\limits_{u\in\Phi_N}T^{P(u)}f-\frac{1}{|\Phi_N|}\sum\limits_{u\in\Phi_N}T^{P(u)}f^{m} \bigg\Vert_{L^2(\mu)}+\varepsilon.
\end{align}
\end{footnotesize}

Let $E=\{(u,v)\in\Z^s\times\Z^s\colon|P(u)-P(v)|\leq m\}$, then by Lemma \ref{density}, for the F$\phi$lner sequence $\{\Phi_n':=[-n,n]^{s}\times[-n,n]^s\}$ there exists $\{t_n\}\subset\Z^{2s}$ such that $d^*(E)=\overline{d}_{\{\Phi_n'+t_n\}}(E)$. For any $n\in\N$, we have $\big\vert E\cap(\Phi_n'+t_n)\big\vert\leq (2n+1)^{(2s-1)}(2m+1)k$, where $k$ is the degree of $P$. Then
$$d^*(E)=\overline{d}_{\{\Phi_n'+t_n\}}\leq\limsup\frac{(2n+1)^{(2s-1)}(2m+1)k}{(2n+1)^{2s}}\to 0$$
as $n\to\infty$. This means $\overline{d}_{\{\Phi_n\times\Phi_n\}}(E)=0$. Thus there exists $N_1\in\N$ such that $\frac{|E\cap(\Phi_N\times\Phi_N)|}{|\Phi_N|^2}<\frac{\varepsilon^2}{4}$ holds for any $N>N_1$. Then for every $N>N_1$ we have
\begin{align}\label{e2}
&\bigg\Vert\frac{1}{|\Phi_N|}\sum\limits_{u\in\Phi_N}T^{P(u)}f-\frac{1}{|\Phi_N|}\sum\limits_{u\in\Phi_N}T^{P(u)}f^{m} \bigg\Vert_{L^2(\mu)}^2\nonumber\\
&=\frac{1}{|\Phi_{N}|^2}\sum\limits_{u,v\in\Phi_N}\int(T^{P(u)}f-T^{P(u)}f^m)(T^{P(v)}f-T^{P(v)}f^m) d\mu\nonumber\\
&=\frac{1}{|\Phi_{N}|^2}\sum\limits_{u,v\in\Phi_N}\bigl(A_{uv}+B_{uv}-C_{uv}-D_{uv}\bigr),
\end{align}
where
$$A_{uv}=\int T^{P(u)}f\cdot T^{P(v)}fd\mu,$$
$$B_{uv}=\int T^{P(u)}f^m\cdot T^{P(v)}f^{m}d\mu,$$
$$C_{uv}=\int T^{P(u)}f\cdot T^{P(v)}f^md\mu,$$
$$D_{uv}=\int T^{P(u)}f^m\cdot T^{P(v)}fd\mu.$$
For $(u,v)\notin E_N$, if $P(u)>P(v)+m$, we have
\begin{small}
\begin{align*}
A_{uv}&=\int E(T^{P(u)-P(v)}f\cdot f\big|T^{-m}\mathscr{P})d\mu\\
&=\int T^{P(u)-P(v)}f\cdot E(f\big|T^{-m}\mathscr{P})d\mu\\
&=C_{uv},
\end{align*}
\end{small}
and
\begin{small}
\begin{align*}
D_{uv}&=\int E\left(T^{P(u)}f^m\cdot T^{P(v)}f\big|T^{-\left(P(v)+m\right)}\mathscr{P}\right)d\mu\\
&=\int T^{P(u)}f\cdot E(T^{P(v)}f\big|T^{-(P(v)+m))}\mathscr{P})\\
&=B_{uv}.
\end{align*}
\end{small}
Similarly, when $P(v)>P(u)+m$ we have $A_{uv}=D_{uv}$ and $B_{uv}=C_{uv}$. Moreover, since $\Vert f\Vert_{\infty}=1$, we have $|A_{uv}|$, $|B_{uv}|$, $|C_{uv}|$, $|D_{uv}|\leq 1$. Thus we have
\begin{footnotesize}
\begin{align*}
(\ref{e2})&\leq\frac{1}{|\Phi_N|^2}\left(\bigg|\sum\limits_{(u,v)\in F_N}\left(A_{uv}+B_{uv}-C_{uv}-D_{uv}\right)\bigg|+\sum\limits_{(u,v)\in E_N}\left(A_{uv}+B_{uv}-C_{uv}-D_{uv}\right)\right)\\
&=\frac{1}{|\Phi_N|^2}\left(\sum\limits_{(u,v)\in E_N}\big\vert A_{uv}+B_{uv}-C_{uv}-D_{uv}\big\vert\right)\\
&\leq4\frac{1}{|\Phi_N|^2}\leq\varepsilon^2,
\end{align*}
\end{footnotesize}
Where $F_N=(\Phi_N\times\Phi_N)\setminus E$ and $E_N=E\cap(\Phi_N\times\Phi_N)$. Then $(\ref{e1})\leq 2\varepsilon$ for any $N>N_1$. Thus the conclusion holds for
$\mathcal{P}$-measurable functions in $L^{\infty}(\mu)$ because of the arbitrary of $\varepsilon$, also for $T^r\mathcal{P}$-measurable functions, $r\in\N$, for $\mu$ is $T$-invariant. For general functions $f\in L^{\infty}(\mu)$, there exists $T^k\mathcal{P}$-measurable functions $f_{k}\in L^{\infty}(\mu)$ which  satisfy the conclusion and converge to $f$ when $k\to\infty$, and we can know this conclusion holds for $f$. Thus the result holds for all functions in $L^\infty(\mu)$.
\end{proof}

\begin{lem}\label{tauup-mu}Let $s\in\N$, $P:\Z^s\to\Z$ be a non-constant polynomial, $(X,\mathscr{B}_X,\mu,T)$ be a measure preserving system with the Pinsker $\sigma$-algebra $P_{\mu}(T)$, and $\mu=\int \mu_xd\mu(x)$ be the disintegration of $\mu$ over $P_{\mu}(T)$. Then for the disintegration of $\mu$ as in Proposition \ref{disintegration}, $\mu=\int \tau_xd\mu(x)$, and every $f\in C(X)$ we have
$$\int fd\tau_x=\int\int fd\tau_yd\mu_x(y),$$
holds for $\mu$-a.e. $x\in X$.
\end{lem}
\begin{proof}Let $\{\Phi_n\}_{n=1}^{\infty}$ be a  F$\phi$lner sequence of $\Z^s$. By Proposition \ref{disintegration}, there exists $\{N_i\}_{i=1}^{\infty}\subseteq\N$ such that
$$\frac{1}{|\Phi_{N_i}|}\sum\limits_{u\in\Phi_{N_i}}f(T^{P(u)}x)\to\int fd\tau_x$$
holds as $i\to\infty$ for $\mu$-a.e. $x\in X$ . Combining with Theorem \ref{characteristic} we have
$$\frac{1}{|\Phi_{N_i}|}\sum\limits_{u\in\Phi_{N_i}}E(f|P_{\mu}(T))(T^{P(u)}x)\to\int fd\tau_x$$
in $L^2(\mu)$ as $i\to\infty$. Thus $\int fd\tau_x$ is $P_{\mu}(T)$-measurable, this means
$$\int fd\tau_x=E(\int fd\tau_y|P_{\mu}(T))(x)=\int\int fd\tau_yd\mu_x.$$
\end{proof}
\section{Proof of Theorem \ref{mainthem1}}
In this section, we will give the proof of Theorem \ref{mainthem1}. To begin with, we introduce some definitions and properties.

Let $(X,T)$ be a topological dynamical system with the metric $\rho$. For a point $x\in X$, the \emph{stable set of $x$} is defined as
$$W^s(x,T)=\{y\in X: \lim\limits_{k\to\infty}\rho(T^kx,T^ky)=0\},$$
and the  \emph{unstable set of $x$} is defined as
$$W^u(x,T)=\{y\in X:\lim\limits_{k\to\infty}\rho(T^{-k}x,T^{-k}y)=0\}.$$
In \cite{Huang-Xu-Yi}, the authors showed the following results.
\begin{thm}\label{W^s(x,T)dense}(\cite[proof of Theorem 1]{Huang-Xu-Yi}) Let $(X,\mathscr{B},\mu,T)$ be a ergodic system with $h_{\mu}(T)>0$, $P_{\mu}(T)$ be the Pinsker $\sigma$-algebra of $(X,\mathscr{B},\mu,T)$ and $\mu=\int\mu_xd\mu(x)$ be the disintegration of $\mu$ over $P_{\mu}(T)$, then for $\mu$-a.e. $x\in X$ one has
$$\overline{W^s(x,T)\cap \supp(\mu_x)}=\supp(\mu_x)\ \  and\ \ \overline{W^u(x,T)\cap \supp(\mu_x)}=\supp(\mu_x).$$
\end{thm}

Let $(X,\mathscr{B}_X,\mu)$ be a Borel probability space, $\mathscr{A}\subseteq\mathscr{B}_X$ a $\sigma$-algebra, and $\mu=\int \mu_x^{\mathscr{A}} d\mu(x)$ be the disintegrated over $\mathscr{A}$ of $\mu$. $\mu\times_{\mathscr{A}}\mu$ is a measure on $(X\times X,\mathscr{B}_X\times\mathscr{B}_X)$ defined by
\[\mu\times_{\mathscr{A}}\mu(A)=\int \mu_x^{\mathscr{A}}\times\mu_x^{\mathscr{A}}(A)d\mu(x).
\]
The following theorem is a classic result(see \cite[Theorem 0.4(iii)]{Danilenko} and \cite[lemma 4.2]{Huang-Xu-Yi} see also
\cite[Theorem 4]{Glasner-Thouvenot-Weiss} for free action).
\begin{thm}\label{Pinsker algebra}Let $(X,\mathscr{B},\mu,T)$ be a ergodic system with the Pinsker $\sigma$-algebra $P_{\mu}(T)$. If $\lambda=\mu\times_{P_{\mu}(T)}\mu$ and $\pi:X\times X\to X$ is the canonical projection to the first factor, then $P_{\lambda}\left(T\times T\vert \pi^{-1}P_{\mu}(T)\right)=\pi^{-1}(P_{\mu}(T))$(mod $\lambda$).
\end{thm}

Note that, under the above settings, for every $A\in P_{\lambda}\left(T\times T\vert\pi^{-1}(P_{\mu}(T))\right)$, there exists $A_0\in P_{\mu}(T)$ such that $A=\pi^{-1}(A_0)=A_0\times X\ (mod\ \lambda)$. Then one has
\begin{align*}
h_{\lambda}(T\times T,\{A,(X\times X)\setminus A\})&=h_{\lambda}(T\times T, \{A_0\times X, (X\setminus A_0)\times X\})\\
&=h_{\mu}(T,\{A_0, X\setminus A_0\})=0.
\end{align*}
This means $A\in P_{\lambda}(T\times T)$. Thus $P_{\lambda}\left(T\times T\right)=\pi^{-1}(P_{\mu}(T))$(mod $\lambda$).

Now we are ready to give the proof Theorem \ref{mainthem1}.
\begin{proof}[Proof of Theorem \ref{mainthem1}]Since $h_{top}(X,T)>0$, there exists $\mu\in \mathcal{M}^e(X,T)$ such that $h_{\mu}(X,T)>0$. Let $P_{\mu}(T)$ be the Pinsker $\sigma$-algebra and $\mu=\int \mu_xd\mu(x)$ be the disintegration of $\mu$ over $P_{\mu}(T)$. We set $\lambda= \mu\times_{P_{\mu}(T)}\mu$, then for the measure preserving system $(X\times X,\mathscr{B}_X\times\mathscr{B}_X,\lambda,T\times T)$, by Theorem \ref{disintegration}, there exist $\{N_i\}_{i=1}^{\infty}$ and a disintegration $\lambda=\int \tau_Zd\lambda(z)$ of $\lambda$, such that
$$\frac{1}{|\Phi_{N_i}|}\sum\limits_{u\in\Phi_{N_i}}f(T^{P{u}}x_1,T^{P(u)}x_2)\to\int fd\tau_{(x_1,x_2)}$$
holds for every $f\in C(X\times X)$ and $\lambda$-a.e. $(x_1,x_2)$.

By Theorem \ref{Pinsker algebra}, we know the Pinsker $\sigma$-algebra of $(X\times X,\mathscr{B}_X\times \mathscr{B}_X,\lambda, T\times T)$ is $\pi^{-1}(P_{\mu}(T))$, where $\pi:X\times X\to X$, $(x_1,x_2)\to x_1$ is the canonical projection to the first coordinate. Thus $\lambda=\int \mu_x\times\mu_xd\mu(x)$ can be also regarded as the disintegration of $\lambda$ over $\pi^{-1}(P_{\mu}(T))$. Then by Lemma \ref{tauup-mu} one has
$$\int fd\tau_{(z_1,z_2)}=\int\int fd\tau_{(y_1,y_2)}d\mu_{z_1}\times\mu_{z_1}(y_1,y_2)$$
for every $f\in C(X^2)$ and $\lambda$-a.e. $(z_1,z_2)\in x^2$. This means $\int fd\tau_{(z_1,z_2)}$ is constant for $\supp(\mu_x)\times \supp(\mu_x)$-a.e. $(z_1,z_2)$ and $\mu$-a.e. $x\in X$.

We set $f_0(x_1,x_2)=\rho(x_1,x_2)$, then $f_0(x_1,x_2)>0$ for every $(x_1,x_2)\notin \Delta_X:=\{(x,x)\colon x\in X\}$. Since $\lambda(\Delta_X)=0$, we have
$$0<\int f_0d\lambda=\int\int\int f_0d\tau_{(y_1,y_2)}d\mu_z\times\mu_zd\mu(z).$$
Then there exist a subset $X_2$ of $X$ with $\mu(X_2)>0$ and a constant $c>0$ such that
$\int\int f_0d\tau_{(y_1,y_2)}d\mu_z\times\mu_z(y_1,y_2)>c$ for any $z\in X_2$. Thus $\int f_0d\tau_{(y_1,y_2)}>c$ for $\mu_z\times\mu_z$-a.e. $(y_1,y_2)$. This is
$$\frac{1}{|\Phi_{N_i}|}\sum\limits_{u\in\Phi_{N_i}}\rho(T^{P(u)}x_1,x_2)\to\int fd\tau_{x_1,x_2}>c>0$$
for $\mu_z\times\mu_z$-a.e. $(x_1,x_2)\in X\times X$ and any $z\in X_2$. We set
$$A=\{(x_1,x_2)\in X\times X: \limsup_{N\to\infty}\frac{1}{|\Phi_{N}|}\sum\limits_{u\in\Phi_{N}}\rho(T^{P(u)}x_1,T^{P(u)}x_2)>c\},$$
it is a $G_\delta$ subset of $X\times X$. And $A\cap(\supp(\mu_z)\times\supp(\mu_z))$ is dense $G_\delta$ subset of $(\supp{\mu_z}\times\supp{\mu_z})$ for every $z\in X_2$.

By Theorem \ref{W^s(x,T)dense}, we know there exists a subset $X_3$ of $X$ with $\mu(X_3)=1$ such that for every $x\in X_3$
$$\overline{W^s(x,T)\cap\supp(u_x)}=\supp(\mu_x).$$
We set
$$B=\{(x_1,x_2)\in X\times X: \liminf_{N\to\infty}\frac{1}{|\Phi_{N}|}\sum\limits_{u\in\Phi_{N}}\rho(T^{P(u)}x_1,T^{P(u)}x_2)=0\},$$
it is a $G_{\delta}$ subset of $X\times X$ and now we shall show $W^s(x,T)\times W^s(x,T)\subseteq B$. Since $X$ is compact, we can assume $diam(X)=1$. For every $y_1,y_2\in W^s(x,T)$ and any $\varepsilon>0$ there exists a $K_0\in\N$ such that for every $k>K_0$ one has $d(T^ky_1,T^ky_2)<\frac{\varepsilon}{2}$.

Let $E=\{u\in\Z^s: P(u)\leq K_0\}$, then by Lemma \ref{density}, for the F$\phi$lner sequence $\{\Phi_n':=[-n,n]^{s}\}$ there exists $\{t_n\}\subset\Z^{s}$ such that $d^*(E)=\overline{d}_{\{\Phi_n'+t_n\}}(E)$. For any $n\in\N$, we have $\big\vert E\cap(\Phi_n'+t_n)\big\vert\leq (2n+1)^{(s-1)}(K_0+1)k$, where $k$ is the degree of $P$. Then
$$d^*(E)=\overline{d}_{\{\Phi_n'+t_n\}}\leq\limsup\frac{(2n+1)^{(s-1)}(K_0+1)k}{(2n+1)^{s}}\to 0$$
as $n\to\infty$. This means $\overline{d}_{\{\Phi_n\}}(E)=0$. Thus there exists $N_1\in\N$ such that $\frac{|E_N|}{|\Phi_N|}<\frac{\varepsilon}{2}$ holds for every $N>N_1$, and
\begin{align*}
&\frac{1}{|\Phi_N|}\sum\limits_{u\in\Phi_N}\rho(T^{P(u)}y_1,T^{P(u)}y_2)\\
&=\frac{1}{|\Phi_N|}\sum\limits_{u\in E_N}\rho(T^{P(u)}y_1,T^{P(u)}y_2)+\frac{1}{|\Phi_N|}\sum\limits_{u\in\Phi\backslash E_N}\rho(T^{P(u)}y_1,T^{P(u)}y_2)\\
&\leq\frac{|E_N|}{|\Phi_N|}+\frac{\varepsilon}{2}\leq\varepsilon,
\end{align*}
Where $E_N=E\cap\Phi_N$. This implies $(y_1,y_2)\in B$, thus $W^s(x,T)\times W^s(x,T)\subseteq B$.

Then $B\cap\big(\supp(\mu_x)\times\supp(\mu_x)\big)$ is a dense $G_{\delta}$ subset of $\supp(\mu_x)\times\supp(\mu_x)$ for every $x\in X_3$. Thus $A\cap B\cap\big(\supp(\mu_x)\times\supp(\mu_x)\big)$ is a dense $G_{\delta}$ subset of $\supp(\mu_x)\times\supp(\mu_x)$ for every $x\in X_2\cap X_3$. Then there exists a Cantor subset $C$ of $\supp(\mu_x)\times\supp(\mu_x)$ with $C\subseteq A\cap B\cap\big(\supp(\mu_x)\times\supp(\mu_x)\big)$ satisfying the requirement of the Theorem \ref{mainthem1}.
\end{proof}

\begin{rem}In the above proof, for any given non-constant polynomial $Q:\Z^s\to\Z^-$, by showing $W^u(x,T)\times W^u(x,T)\subseteq B$ and noticing the fact that $\overline{W^u(x,T)\cap}$ $\overline{\supp(\mu_x)}= \supp(\mu_x)$, we can also show that $B\cap \supp(\mu_x)\times\supp(\mu_x)$ is a dense $G_{\delta}$ subset of $\supp(\mu_x)$. Hence Theorem \ref{mainthem1} also established along $Q$.
\end{rem}
\begin{rem}For a given F$\phi$lner sequence $\{\Phi_n\}_{n=1}^{\infty}$, let $t_n=\min\limits_{u\in\Phi_n}\min\{n_i,u=(n_1,n_2,\dotsc,n_s)\}$. If there exists $t_0\in\N$ such that $t_0\leq t_n$ for every $n\in\N$ (for example $\{[0,n]^s\}_{n=1}^{\infty}$). Then for any $P:\Z^s\to\Z$, Theorem \ref{mainthem1} also established.
\end{rem}

\section{mean li-yorke chaos along polynomials of prime numbers}
For studying an average over the primes, the authors of \cite{Frantzikinakis-Host-Kra} replaced this average with a certain weighted average over the integers, and in the proof of Lemma 1 they gave the following result.
\begin{lem}\cite[Proof of Lemma 1]{Frantzikinakis-Host-Kra}\label{Host-Kra} For any $\varepsilon>0$ there exists $N_0\in\N$, such that for any $N>N_0$ and map $a:\N\to\R$ with $\vert a(n)\vert\leq1$ one has
$$\bigg\vert\frac{1}{\pi(N)}\sum\limits_{p\in\p,p<N}a(p)-\frac{1}{N}\sum\limits_{n=0}^{N-1}\Lambda(n)a(n)\bigg\vert<\varepsilon.$$
where $\lambda\colon\mathbb{Z}\to\mathbb{R}$ is the von Mangoldt function defined by
$$\Lambda(n)=
\begin{cases}
\log p & \text{if $n=p^m$ for some $m\in\mathbb{N}$ and $p\in\mathbb{N}$,}\\
0& \text{otherwise.}
\end{cases}$$
\end{lem}

Then for any $s\in\N$ and map $a:\N^s\to\R$ with $\vert a(n)\vert\leq 1$, one has
\begin{scriptsize}
\begin{align*}
&\bigg\vert\frac{1}{(\pi(N))^s}\sum\limits_{\mbox{\tiny$\begin{array}{c}
p_1,\dotsc p_s\in\p\nonumber\\
p_1,\dotsc,p_s<N\end{array}$}}a(p_1,\dotsc,p_s)-\frac{1}{N^s}
\sum\limits_{n_1,\dotsc,n_s=0}^{N-1}\Lambda(n_1)\dotsc\Lambda(n_s)a(n_1,\dotsc,n_s)\bigg\vert\\
&\leq\sum_{\ell=1}^{s}\left(\frac{1}{N^{\ell-1}}\sum\limits_{n_1,\dotsc,n_{\ell-1}=0}^{N-1}\Lambda(n_1)\dotsc\Lambda(n_{\ell-1})
\frac{1}{\pi(N)^{s-\ell-1}}\sum\limits_{\mbox{\tiny$\begin{array}{c}
p_{\ell+1},\dotsc p_s\in\p\nonumber\\
p_{\ell+1},\dotsc,p_s<N\end{array}$}}
\bigg\vert A_{n_1,\dotsc,n_{\ell-1},p_{\ell+1},\dotsc,p_s}(N)\bigg\vert\right)
\end{align*}
\end{scriptsize}
where
\begin{align*}
A_{n_1\dotsc n_{\ell-1},p_{\ell+1}\dotsc p_s}(N)=\bigg\vert\frac{1}{\pi(N)}\sum\limits_{p_\ell\in\p,p_{\ell}\leq N}
a(n_1\dotsc n_{\ell-1},p_{\ell}\dotsc p_s)-\\
\frac{1}{N}\sum\limits_{n_{\ell}=0}^{N-1}\lambda(n_{\ell}) a(n_1\dotsc n_{\ell},p_{\ell+1}\dotsc p_s)\bigg\vert.
\end{align*}
By Lemma \ref{Host-Kra}, for large enough $N\in\N$ and $d_1,\dotsc,d_{\ell-1}$, $p_{\ell},\dotsc,p_s\in\N$, we have $\vert A_{n_1,\dotsc,n_{\ell-1},p_{\ell+1},\dotsc,p_s}(N)\vert<\frac{\varepsilon}{2s}$ for every $n_1,\dotsc,n_{\ell-1},p_{\ell+1},\dotsc,p_s$ and $\ell\in\N$. Then by the well known fact that $\Lambda$ has mean one, for large enough $N$ one has $\frac{1}{N^{\ell-1}}\sum\limits_{n_1,\dotsc,n_{\ell-1}=0}^{N-1}\Lambda(n_1)\dotsc\Lambda(n_{\ell-1})\leq 2$ for every $\ell=1,2,\dotsc,s$. Concluding the above discussion, we have the following lemma.

\begin{lem}\label{P-host-kra} For any $s\in\N$ and $\varepsilon>0$, there exists $N_0\in\N$ such that for large enough $N$ and map $a:\N^s\to[-1,1]$ one has
$$\bigg\vert\frac{1}{(\pi(N))^s}\sum\limits_{\mbox{\tiny$\begin{array}{c}
p_1,\dotsc p_s\in\p\nonumber\\
p_1,\dotsc,p_s<N\end{array}$}}a(p_1,\dotsc,p_s)-\frac{1}{N^s}
\sum\limits_{n_1,\dotsc,n_s}^{N-1}\Lambda(n_1)\dotsc\Lambda(n_s)a(n_1,\dotsc,n_s)\bigg\vert<\varepsilon.$$
\end{lem}

In the proof of the \cite[Theorem 3]{Frantzikinakis-Host-Kra}, the authors proved the following lemma.
\begin{lem}\label{P-host-kra2}For any $\varepsilon>0$ and measure preserving system $(X,\mathscr{B}_X,T,\mu)$, there exists $W_0$ and $N_0\in\N$ such that for any $N>N_0$, and $a\colon X\times\N\to\R\in C(X,\N)$ with $\vert a(x,n)\vert\leq 1$, we have
\begin{scriptsize}
\begin{align*}&\bigg\Vert \frac{1}{[W_0N/3]}\sum\limits_{0\leq n<[W_0N/3]}\big(\Lambda(n)a(x,n)\big)-\frac{1}{\phi(W_0)}\sum\limits_{\mbox{\tiny$\begin{array}{c}
0\leq r<W_0,\nonumber\\
(r,W_0)=1\end{array}$}}\frac{1}{[N/3]}\sum\limits_{0\leq n<[N/3]}a(x,W_0n+r)\bigg\Vert_{L^2(\mu)}<\varepsilon,
\end{align*}
\end{scriptsize}
where $\phi$ is the Euler function.
\end{lem}

\begin{thm}\label{P-L^2} Let $s\in\N$, $P:\N^s\to\N$ be a non-constant integer polynomial, then for any measure preserving system $(X,\mathscr{B}_X,\mu,T)$ and $f\in L^{\infty}(X)$, the average
$$\frac{1}{\pi(N)^s}\sum\limits_{\mbox{\tiny$\begin{array}{c}
0\leq p_1\dotsc p_s< N,\\
p_1\dotsc p_s\in\p\end{array}$}}T^{P(p_1,p_2,\dotsc,p_s)}f$$
convergence in $L^2(\mu)$ as $N\to\infty$.
\end{thm}
\begin{proof}Without loss of generality, we can assume $\Vert f\Vert_{\infty}\leq 1$. For any $x\in X$, we set $a_x:\N^s\to\R$ as $a_x(n_1,\dotsc,n_s)=f(T^{P(n_1,\dotsc,n_s)}x)$. Then by Lemma \ref{P-host-kra}, it is suffice to show the average
$$\frac{1}{N^s}\sum\limits_{0\leq n_1\dotsc n_s<N}\bigg(\Lambda(n_1)\dotsc\Lambda(n_s)T^{P(n_1,n_2,\dotsc,n_s)}f\bigg)$$
convergence in $L^2(\mu)$ as $N\to\infty$.
\begin{footnotesize}
\begin{align}\label{P-e3}
&\bigg\Vert \frac{1}{[WN/3]^s}\sum\limits_{n_1,\dotsc,n_s=0}^{[WN/3]-1}\bigg(\Lambda(n_1)\dotsc\Lambda(n_s)a_x(n_1\dotsc n_s)\bigg)\nonumber\\
 &\ \ \ \ \ \ \ \  \ \ \ \   -\frac{1}{\phi(W)^s}\sum\limits_{\mbox{\tiny$\begin{array}{c}
0\leq r_1,\dotsc,r_s<W,\nonumber\\
(r_i,W)=1\end{array}$}}\frac{1}{[N/3]^s}\sum\limits_{n_1,\dotsc,n_s=0}^{[N/3]-1}a_x(Wn_1+r_1,\dotsc,Wn_s+r_s)\bigg\Vert_{L^2(\mu)}\nonumber\\
&\leq \sum\limits_{\ell=1}^{s}\frac{1}{\phi(W)^{\ell-1}}\sum\limits_{\mbox{\tiny$\begin{array}{c}
0\leq r_1,\dotsc,r_{\ell-1}<W,\\
(r_i,W)=1\end{array}$}}\frac{1}{[N/3]^{\ell-1}}\sum\limits_{n_1,\dotsc,n_{\ell-1}=0}^{[N/3]-1}
B_{r_1\dotsc r_{\ell-1}}^{n_1\dotsc n_{\ell-1}}([WN/3])
\end{align}
\end{footnotesize}
where $\phi$ the Euler function,
\begin{footnotesize}
$$B_{r_1\dotsc r_{\ell-1}}^{n_1\dotsc n_{\ell-1}}([WN/3])=\frac{1}{[WN/3]^{s-\ell}}\sum\limits_{n_{\ell+1\dotsc n_s}=0}^{[WN/3]-1}\Lambda(n_{\ell+1})\dotsc\Lambda(n_s)D_{r_1\dotsc r_{\ell-1}}^{n_1\dotsc n_{\ell-1},n_{\ell+1},\dotsc,n_s}([WN/3]),$$
\end{footnotesize}
and
\begin{scriptsize}
\begin{align*}
D_{r_1,\dotsc,r_{\ell-1}}^{n_1,\dotsc,n_{\ell-1},n_{\ell+1},\dotsc,n_s}([WN/3])=\bigg\Vert\frac{1}{[WN/3]}\sum\limits_{n_\ell=0}^{[WN/3]-1}
\Lambda(n_\ell)a_x(Wn_1+r_1,\dotsc Wn_{\ell-1}+r_{\ell-1},n_\ell,\dotsc n_s)\\
-\frac{1}{\phi(W)}\sum\limits_{\mbox{\tiny$\begin{array}{c}
0\leq r_{\ell}<W,\nonumber\\
(r_\ell,W)=1\end{array}$}}\frac{1}{[N/3]}\sum\limits_{n_\ell=0}^{[N/3]-1}a_x(Wn_1+r_1,\dotsc,Wn_\ell+r_{\ell},n_{\ell+1},\dotsc,n_s)\bigg\Vert_{L^2(\mu)}.
\end{align*}
\end{scriptsize}

For any $\varepsilon>0$, by the Lemma \ref{P-host-kra2}, there exists $W_0\in\N$ and such that for $N\in\N$ large enough, one has
$$D_{r_1,\dotsc,r_{\ell-1}}^{n_1,\dotsc,n_{\ell-1},n_{\ell+1},\dotsc,n_s}([W_0N/3])<\frac{\varepsilon}{2s},$$
for every $r_1,\dotsc,r_{\ell-1}$ and $n_1,\dotsc,n_{\ell-1},n_{\ell+1},\dotsc,n_s$ and $\ell=0,1,\dotsc,s$. By the fact that $\Lambda$ has mean one, for large enough $N\in\N$ we have
$$\frac{1}{[W_0N/3]^{s-\ell}}\sum\limits_{n_{\ell+1}\dotsc n_s=0}^{[W_0N/3]-1}\Lambda(n_{\ell+1})\dotsc\Lambda(n_s)<2$$
for every $1\leq\ell\leq s$. Thus there exits $W_0\in\N$ and $N_0\in\N$, such that for any $N>N_0$ one has $(\ref{P-e3})<\varepsilon$.

By Theorem \ref{Leibman},
$$\frac{1}{\phi(W_0)^s}\sum\limits_{\mbox{\tiny$\begin{array}{c}
1\leq r_1,\dotsc,r_s<W_0,\nonumber\\
(r_i,W_0)=1\end{array}$}}\frac{1}{[N/3]^s}\sum\limits_{n_1\dotsc n_s=0}^{[N/3]-1}a_x(W_0n_1+r_1,\dotsc,W_0n_s+r_s)$$
converges in $L^2(\mu)$ as $N\to\infty$. We set
$$E([W_0N/3])(x)=\frac{1}{[W_0N/3]^s}\sum\limits_{n_1\dotsc,n_s=0}^{[W_0N/3]-1}\Lambda(n_1)\dotsc\Lambda(n_s)a_x(n_1\dotsc n_s)).$$
Using triangle inequality we have that for large enough $M$, $N$
$$\bigg\Vert E([W_0N/3])(x)-E([W_0M/3])(x)\bigg\Vert<\varepsilon.$$
Since $E([W_0N/3]+i)(x)-E([W_0N/3])(x)\to0$ as $N\to\infty$, for $0\leq i<W_0/3$, we conclude that $E(N)(x)$ is Cauchy. And this finishes our proof.
\end{proof}

Similar with the proof of Proposition \ref{disintegration}, we have the following proposition.
\begin{prop}\label{P-disintegration}Let $s\in\N$ and $P:\N^s\to\N$, then for any measure preserving system $(X,\mathscr{B}_X,\mu,T)$, there exists a disintegration of $\mu$, $\mu=\int\tau_xd\mu(x)$, in the sense that there exist $\{N_i\}_{i=1}^{\infty}$ and a Borel subset $X_0\subseteq X_1$ with $\mu(X_0)=1$, such that
$$\lim\limits_{i\to\infty}\frac{1}{\pi(N_i)^s}\sum\limits_{\mbox{\tiny$\begin{array}{c}
1\leq p_1\dotsc,p_s\leq N_i,\\
p_i\in\p\end{array}$}}f(T^{P(p_1,p_2,\dotsc,p_s)}x)=\int fd\tau_x$$
holds for every $x\in X_0$ and $f\in C(X)$, and
$$\int\int fd\tau_xd\mu(x)=\int fd\mu.$$
\end{prop}

Note that, for any $s\in\N$, non-constant polynomial $P:\N^s\to\Z$ with degree $m$ and $t_0<t_1\in\Z$, let
$$E(N)=\{(p_1,p_2,\dotsc,p_s)\in\p^s\cap[1,N]^s:\ \ t_0<P(p_1,p_2,\dotsc,p_s)<t_1\}.$$
We have
$$\vert E(N)\vert\leq\pi(N)^{(s-1)}(t_1-t_0)m.$$
Thus $\frac{\vert E(N)\vert}{\pi(N)^s}\to 0$ as $N\to\infty$. Then following the proof of Theorem \ref{characteristic} and \ref{tauup-mu}, one has the following results.

\begin{thm}\label{P-characteristic}Let $s\in\N$, $P\colon\N^s\to\N$, $(X,\mathscr{B}_X,\mu,T)$ be a measure preserving system, and $P_{\mu}(T)$ be its Pinsker $\sigma$-algebra. Then for any $f\in L^{\infty}(\mu)$
$$\frac{1}{\pi(N)^s}\sum\limits_{\mbox{\tiny$\begin{array}{c}
0\leq p_1,\dotsc,p_s<N,\\
p_i\in\p\end{array}$}}\bigg(T^{P(p_1,\dotsc,p_s)}f-T^{P(p_1,\dotsc,p_s)}E(f|P_{\mu}(T)\bigg)\to0$$
in $L^2(\mu)$ as $N\to\infty$.
\end{thm}
\begin{lem}\label{P-constant}Let $(X,\mathscr{B}_X,\mu,T)$ be a measure preserving system, $P_{\mu}(T)$ be its Pinsker $\sigma$-algebra and $\mu=\int\mu_xd\mu(x)$ be the disintegration of $\mu$ over $P_{\mu}(T)$. If $\mu=\int\tau_xd\mu(x)$ is the disintegration of $\mu$ as in the Proposition \ref{P-disintegration}, then for every $f\in C(X)$ one has
$$\int f\tau_x=\int\int f\tau_yd\mu_x(y)$$
for $\mu$-a.e. $x\in X$.
\end{lem}

The proof of Theorem \ref{mainthm2} can be obtained by following the proof of Theorem \ref{mainthem1} and combining with the above results.

\bibliographystyle{amsplain}

\end{document}